\documentclass{base_class}

\title[]{On a generalization of Godbersen's conjecture} 

\author{Jan Kotrbat{\'y}}\thanks{Supported by Charles University grants PRIMUS/24/SCI/009 and UNCE/24/SCI/022.}
 
\address{Charles University, Faculty of Mathematics and Physics, Mathematical Institute of Charles University, Sokolovsk\'a 49/83, 186 00 Prague, Czechia}
\email{kotrbaty@karlin.mff.cuni.cz}
\date{\today}

\begin{document}

\maketitle

\begin{abstract} 
The long-standing Godbersen's conjecture asserts that the Rogers--Shephard inequality for the volume of the difference body is refined by an inequality for the mixed volume of a convex body and its reflection about the origin. The conjecture is known in several special cases, notably for anti-blocking convex bodies. In this note, we propose a generalization of Godbersen's conjecture that refines Schneider's generalization of the Rogers--Shephard inequality to higher-order difference bodies  and prove our conjecture for anti-blocking convex bodies. Moreover, we relate the conjectured inequality to the higher-rank mixed volume defined by the author and Wannerer which leads to an equivalent formulation in terms of the Alesker product of smooth, translation invariant valuations.
\end{abstract}

\section{Introduction}

The difference body $DK$ of a convex body $K\subset\RR^n$ is, by definition, the Minkowski sum of $K$ with its reflection $-K$ about the origin. Due to a classical result of Rogers and Shephard \cite{RogersShephard57,RogersShephard58}, the volume of the difference body is bounded above in terms of the volume of $K$ as follows:
\begin{align}
\label{eq:RS}
\vol_n(DK)\leq{2n\choose n}\vol_n(K).
\end{align}
Besides bodies with empty interior, equality in \eqref{eq:RS} is attained if and only if $K$ is a simplex.

When restricted to convex bodies, the Lebesgue measure $\vol_n$ polarizes with respect to the Minkowski addition. This yields a unique $n$-linear symmetric function $V$, called mixed volume, satisfying $V(K,\dots,K)=\vol_n(K)$, see Section \ref{s:background} for an explicit definition. Expanding the left-hand side of \eqref{eq:RS} by linearity and the right-hand side using the identity ${2n\choose n}=\sum_{k=0}^n{n\choose k}^2$, one obtains an inequality between sums of length $n+1$. It was conjectured by Godbersen \cite{Godbersen} in 1938 that the inequality in fact holds term-wise. More precisely, denoting $k$ copies of the same convex body $K$ by $K[k]$, one has the following:
\begin{conjecture}[Godbersen]
\label{Godbersen}
For each convex body $K\subset\RR^n$ and each $0\leq k\leq n$ one has
\begin{align}
\label{eq:Godbersen}
V(-K[k],K[n-k])\leq  {n\choose k}\vol_n(K).
\end{align}
Moreover, if $0<k<n$ and $\interior K\neq\emptyset$, equality holds if and only if $K$ is a simplex.
\end{conjecture}
Conjecture \ref{Godbersen} is known only in certain special cases. The case $k\in\{0,n\}$ is trivial. The case $k\in\{1,n-1\}$ follows from a result of Schneider \cite{Schneider09}. For general $k$, the conjecture was verified for bodies of constant width by Godbersen \cite{Godbersen}. Much more recently, Artstein-Avidan, Sadovsky, and Sanyal \cite{ArtsteinSadovskySanyal23} gave a proof for anti-blocking bodies which was generalized to locally anti-blocking bodies by Sadovsky \cite{Sadovsky25}. Artstein-Avidan, Einhorn, Florentin, and Ostrover \cite{ArtsteinETAL15} proved a related statement that in particular implies \eqref{eq:Godbersen} up to a factor $\sqrt n$. Very recently, their result was improved by Artstein-Avidan and Putterman \cite{ArtsteinPutterman24}.

A different generalization of the Rogers--Shephard inequality was considered by Schneider \cite{Schneider70}. Observing that $DK=\{x\in\RR^n\mid K\cap (K+x)\neq\emptyset\}$, Schneider defined for each $p\in\NN$ the higher-order difference body $D_pK=\{(x_1,\dots,x_p)\in(\RR^n)^p\mid K\cap (K+x_1)\cap\cdots\cap (K+x_p)\neq\emptyset\}$
and proved that
\begin{align}
\label{eq:Schneider}
\vol_{pn}(D_pK)\leq {pn+n\choose n}\vol_n(K)^p
\end{align}
with equality if and only if $K$ is a simplex (unless $\interior K=\emptyset$). An alternative proof of \eqref{eq:Schneider} was recently given by Haddad, Langharst, Putterman, Roysdon, and Ye \cite{HaddadETAL23} who also extended Schneider's generalization of the difference body to other important convex-geometric constructions and proved the respective inequalities. Functional versions of these inequalities were subsequently proven by Langharst, Sola, and Ulivelli \cite{LangharstSolaUlivelli24} and by Langharst, Putterman, Roysdon, and Ye \cite{LangharstETAL24}.

It is easily seen that the (original) Minkowski sum presentation of $DK$ generalizes to the higher-order difference bodies. Namely, one has $D_pK=\Delta_pK-K^p$, where $\Delta_p:\RR^n\to(\RR^n)^p$ is the diagonal embedding and $K^p=K\times\cdots\times K$ the $p$-ary Cartesian power of $K$, see Proposition \ref{pro:DpK}. Motivated by this observation, we propose---in the same way as Godbersen's conjecture refines the Rogers--Shephard inequality---the following refinement of Schneider's inequality:
\begin{conjecture}
\label{conj1}
For each convex body $K\subset\RR^n$, $0\leq k\leq n$, and $p\in\NN$ one has
\begin{align}
\label{eq:conj1}
V(-\Delta_pK[k],K^p[pn-k])\leq {n\choose k}\vol_n(K)^p.
\end{align}
Moreover, if $0<k$ and $\interior K\neq\emptyset$, equality holds if and only if $K$ is a simplex.
\end{conjecture}
Observe that Conjecture \ref{conj1} subsumes Godbersen's conjecture for $p=1$. In this connection, let us also mention that a different generalization of Conjecture \ref{Godbersen} was proposed in \cite{ArtsteinETAL15}. The fact that \eqref{eq:conj1} implies \eqref{eq:Schneider} follows at once from multilinearity of the mixed volume and an elementary combinatorial identity, see Section \ref{s:genealogy}. The generalizing procedure can be further iterated using that $K^p=\sum_{j=1}^p\iota_jK$, where $\iota_j:\RR^n\to(\RR^n)^p$ is the inclusion into the $j$-th factor, along with another combinatorial identity. This leads us to the following more fundamental conjecture that strengthens Conjecture \ref{conj1}:
\begin{conjecture}
\label{conj2}
Let $K\subset\RR^n$ be a convex body. For any $p\in\NN$, $0\leq k\leq n$, and $0\leq k_i\leq n$, $i=1,\dots, p$, such that $k_1+\cdots+k_p=k$ one has
\begin{align}
\label{eq:conj2}
V(-\Delta_pK[k],\iota_1K[n-k_1],\dots,\iota_pK[n-k_p])\leq {n\choose k}{k\choose k_1,\dots,k_p}\frac{(n!)^p}{(pn)!}\vol_n(K)^p.
\end{align}
Moreover, if $k_i<n$ for all $i$, $0<k$, and $\interior K\neq\emptyset$, equality holds if and only if $K$ is a simplex.
\end{conjecture}

Our main result establishes these conjectures for the class of anti-blocking convex bodies. Recall that a convex body is called anti-blocking if it is contained in the positive orthant and its projection to each coordinate subspace equals the section with the subspace, see Section \ref{s:background}.

\begin{theorem}
\label{thm:main}
Conjecture \ref{conj2} (and hence Conjecture \ref{conj1}) is true if the convex body $K$ is assumed to be anti-blocking.
\end{theorem}

The proof of Theorem \ref{thm:main} depends on the dual Bollob\'as–Thomason inequality recently proven by Liakopoulos \cite{Liakopoulos19} and on the characterization of its equality cases due to Boroczky, Kalantzopoulos, and Xi \cite{BKX23}. Let us also mention that other geometric inequalities for (locally) anti-blocking bodies were proven by Artstein-Avidan, Sadovsky, and Sanyal \cite{ArtsteinSadovskySanyal23}, Manui, Ndiaye, and Zvavitch \cite{ManuiNdiayeZvavitch24}, and Sanyal and Winter \cite{SanyalWinter25}.

A new, fruitful approach to geometric inequalities was recently discovered to be based on Alesker's algebraic theory of valuations \cite{Alesker01,Alesker04,BernigFu06}. A central object here is the algebra $\Val^\infty(\RR^n)$ of smooth, translation invariant valuations, see Section \ref{s:background} for a precise definition. As prominent examples of such valuations are mixed volumes with fixed convex bodies that have smooth boundary and positive curvature, it turns out that various inequalities between mixed volumes can be formulated in terms of the Alesker product or the Bernig--Fu convolution on $\Val^\infty(\RR^n)$. A notable example is the array of quadratic inequalities broadly generalizing the Alexandrov--Fenchel inequality discovered recently by Bernig, the author, and Wannerer \cite{KW23,BKW23} by proving the hard Lefschetz theorem and Hodge--Riemann relations for the algebra $\Val^\infty(\RR^n)$. For other results of this type see \cite{K21,KW22,Alesker22}.

Motivated by these developments, we observe that the left-hand side of \eqref{eq:conj2} is an instance of the mixed volume of rank $p$ defined by the author and Wannerer \cite{KW22}. Using that higher-rank mixed volumes appear as coefficients in the Alesker product of (the usual) mixed volumes, we finally show that the first part of Conjecture \ref{conj2} is equivalent to the following:
\begin{conjecture}
\label{conj3}
Let $K\subset\RR^n$ be a convex body with smooth boundary and positive curvature, and $n=k_0+\cdots+k_p$ a partition into non-negative integers. The Alesker product of valuations $\phi_i=(\Cdot[k_i],K[n-k_i])$, $i=0,\dots,p$, satisfies
\begin{align}
\label{eq:conj3}
\phi_0\cdots\phi_p\leq\vol_n(K)^p\vol_n.
\end{align}
\end{conjecture}
Notice that if Conjecture \ref{conj2} is true, equality is never attained in \eqref{eq:conj3}. Nonetheless, we believe  the reformulation of the inequality \eqref{eq:conj2} in terms of the Alesker product can still be instructive for several reasons, see the discussion at the end of Section \ref{s:Alesker}.

\subsection*{Acknowledgments} We are grateful to Davide Ravasini for inspiring discussions that provided us with the initial motivation to formulate the above conjectures and for many useful comments during the preparation of the manuscript. We also thank to Andreas Bernig for useful discussions and the anonymous referees for their careful reading and many helpful suggestions.

\section{Background}
\label{s:background}

\subsection{Anti-blocking convex bodies}

Let us first fix some notation. Throughout the text we assume $n\geq 2$. The standard basis of $\RR^n$ will always be denoted by $e_1,\dots,e_n$ and the canonical Lebesgue measure by $\vol_n$. For any subset $\sigma\subset\{1,\dots,n\}$ let $E_\sigma=\linspan\{e_i\mid i\in\sigma\}$ be the corresponding coordinate subspace and $C_\sigma=\{\sum_{i\in\sigma}\lambda_ie_i\mid \lambda_i>0\}$ the (relatively) open positive orthant in $E_\sigma$. Notice that we follow the convention $E_\emptyset=C_\emptyset=\{0\}$. For a convex body, i.e., a compact convex subset $K\subset \RR^n$ we set $K_\sigma=K\cap C_\sigma$. Finally, we denote $\sigma^c=\{1,\dots,n\}\setminus\sigma$.

A convex body $K\subset\RR^n$ is called locally anti-blocking if for each $\sigma\subset\{1,\dots,n\}$ one has $P_{E_\sigma}K=K\cap E_\sigma$, where $P_{E_\sigma}:\RR^n\to E_\sigma$ is the orthogonal projection. A locally anti-blocking convex body $K\subset\RR^n$ is called anti-blocking if $K\subset\RR^n_{\geq0}$. Equivalently, $K\subset\RR^n_{\geq0}$ is anti-blocking if and only if for any $y\in K$ and $x\in\RR^n_{\geq0}$ such that $x_i\leq y_i$, $i=1,\dots,n$, it holds that $x\in K$. Observe that for any anti-blocking convex body $K\subset\RR^n$ one has the (disjoint) decomposition
\begin{align}
\label{eq:Ksigma}
K=\bigcup_\sigma K_\sigma,
\end{align}
where the union is over all subsets $\sigma\subset\{1,\dots,n\}$.

\subsection{Dual Bollob\'as–Thomason inequality}
\label{ss:dualBT}
Let $p,q\in\NN$. We say that a collection of subsets $\sigma_1,\dots,\sigma_q\subset\{1,\dots,n\}$ form a $p$-uniform cover of $\{1,\dots,n\}$ if each $j\in\{1,\dots,n\}$ is contained in exactly $p$ of the the sets $\sigma_1,\dots,\sigma_q$. Given a $p$-uniform cover $\sigma_1,\dots,\sigma_q\subset\{1,\dots,n\}$, the collection $\tilde\sigma_1,\dots,\tilde\sigma_r$ of all sets of the form $\bigcap_{i=1}^q\sigma_i^{\epsilon(i)}$, where $\epsilon(i)\in\{0,1\}$ and $\sigma_i^0=\sigma_i$ and $\sigma_i^1=\sigma_i^c$, form a $1$-uniform cover of $\{1,\dots,n\}$. We call $\tilde\sigma_1,\dots,\tilde\sigma_r$ the $1$-uniform cover induced by $\sigma_1,\dots,\sigma_q$.

\begin{theorem}[Liakopoulos {\cite[Theorem 1.2]{Liakopoulos19}}]
\label{thm:dualBT}
Let $K\subset\RR^n$ be a convex body with $0\in\interior K$ and $\sigma_1,\dots,\sigma_q$ a $p$-uniform cover of $\{1,\dots,n\}$. Then
\begin{align}
\label{eq:dualBT}
\vol_n(K)^p\geq\frac{\prod_{i=1}^q|\sigma_i|!}{(n!)^p}\prod_{i=1}^q\vol_{|\sigma_i|}(K\cap E_{\sigma_i}).
\end{align} 
\end{theorem}

\begin{theorem}[Boroczky--Kalantzopoulos--Xi {\cite[Theorem 12]{BKX23}}]
\label{thm:equality}
Let $K\subset\RR^n$ be a convex body with $0\in\interior K$ and $\sigma_1,\dots,\sigma_q$ a $p$-uniform cover of $\{1,\dots,n\}$. Then equality holds in \eqref{eq:dualBT} if and only if $K=\conv\{K\cap E_{\tilde\sigma_i}\mid i=1,\dots,r\}$, where $\tilde\sigma_1,\dots,\tilde\sigma_r$ is the $1$-uniform cover  induced by $\sigma_1,\dots,\sigma_q$.
\end{theorem}

\subsection{Mixed volume}

Consider an $n$-dimensional vector space $W$ with a fixed  Lebesgue measure $\vol_W$. The mixed volume of convex bodies $K_1,\dots,K_n\subset W$ is defined as
\begin{align*}
V(K_1,\dots,K_n)=\frac{1}{n!}\frac{\partial^n}{\partial \lambda_1\cdots\partial\lambda_n}\vol_W(\lambda_1 K_1+\cdots+\lambda_n K_n).
\end{align*}
Below the mixed volume will also be denoted by $V_W$ to emphasize the underlying space. By a classical result of Minkowski, $\vol_W(\lambda_1 K_1+\cdots+\lambda_n K_n)$ is in fact an $n$-homogeneous polynomial in $\lambda_i>0$. The mixed volume satisfies for any convex bodies $L,K_1,\dots,K_n\subset W$, permutation $\rho$ of $\{1,\dots,n\}$, $\lambda>0$, and $T\in\GL(W)$ the following properties:
\begin{align*}
V(L,\dots,L)&=\vol_W(L),\\
V\big(K_{\rho(1)},\dots,K_{\rho(n)}\big)&=V(K_1,\dots,K_n),\\
V(K_1+\lambda L,K_2,\dots,K_n)&=V(K_1,\dots,K_n)+\lambda V(L,K_2,\dots,K_n),\\
V(TK_1,\dots,TK_n)&=|\det T|\,V(K_1,\dots,K_n),\\
V(K_1,\dots,K_n)&\geq0.
\end{align*}
Moreover, $V(K_1,\dots,K_n)>0$ if and only if there are segments $S_i\subset K_i$, $i=1,\dots,n$, with linearly independent directions. For reference and more on mixed volumes, see \cite[Chapter 5]{Schneider2014}.

Consider an exact sequence $0\longrightarrow W_1\stackrel{f}{\longrightarrow} W_2\stackrel{g}{\longrightarrow} W_3\longrightarrow 0$ of finite-dimensional vector spaces and fix Lebesgue measures $\vol_{W_1}$ and $\vol_{W_2}$ on $W_1$ and $W_2$, respectively. There is a canonical Lebesgue measure $\vol_{W_3}$ on $W_3$ given as follows: Let $s:W_3\to W_2$ be an arbitrary linear map such that $W_2=f(W_1)\oplus s(W_3)$. Then $\vol_{W_3}$ is the unique Lebesgue measure on $W_3$ such that $\vol_{W_2}=f_*\vol_{W_1}\otimes s_*\vol_{W_3}$.

The following well-known fact about mixed volumes will be crucial for the proof of our main theorem, cf. \cite[Lemma 5.1]{KW22}, \cite[Lemma 2.6.1]{Alesker11}, or \cite[Theorem 5.3.1]{Schneider2014}.

\begin{lemma}
\label{lem:exact}
Let $0\longrightarrow W_1\stackrel{f}{\longrightarrow} W_2\stackrel{g}{\longrightarrow} W_3\longrightarrow 0$ be an exact sequence of vector spaces with $\dim W_1=k$ and $\dim W_2=n$. Let $\vol_{W_1}$ and $\vol_{W_2}$ be Lebesgue measures on $W_1$ and $W_2$, and $\vol_{W_3}$ the induced Lebesgue measure on $W_3$. The corresponding mixed volumes satisfy
\begin{align*}
{n\choose k}V_{W_2}\big(K_1,\dots,K_{n-k},f(L_1),\dots,f(L_k)\big)=V_{W_3}\big(g(K_1),\dots,g(K_{n-k})\big)\,V_{W_1}(L_1,\dots,L_k)
\end{align*}
for any convex bodies $K_1,\dots,K_{n-k}\subset W_2$ and $L_1,\dots, L_k\subset W_1$.
\end{lemma}

\subsection{Alesker product of smooth valuations}

To conclude the background section, let us collect several facts from the algebraic theory of valuations on convex bodies, following \cite{AleskerFu2014}.

A valuation is a (complex-valued) function $\phi$ on the space of convex bodies in $\RR^n$ satisfying
$$
\phi(K\cup L)=\phi(K)+\phi(L)-\phi(K\cap L)
$$
for any convex bodies $K,L\subset\RR^n$ such that $K\cup L$ is convex. The space $\Val(\RR^n)$ of all continuous (with respect to the Hausdorff metric), translation invariant valuations is a Banach space. It contains a natural dense subspace $\Val^\infty(\RR^n)$ of smooth, translation invariant valuations. By a recent result of Knoerr \cite{Knoerr23}, 
$$
\Val^\infty(\RR^n)=\linspan\{V(\Cdot[k],K_1,\dots,K_{n-k})\mid 0\leq k\leq n\text{ and }K_1,\dots,K_{n-k}\in\calK_+^\infty(\RR^n) \},
$$
where $\calK_+^\infty(\RR^n)$ is the set of convex bodies in $\RR^n$ with smooth boundary and positive curvature. Observe that one has the grading $\Val^\infty(\RR^n)=\bigoplus_{k=0}^n\Val_k^\infty(\RR^n)$, where $\Val_k^\infty(\RR^n)$ is the subspace of $k$-homogeneous valuations, and that $\Val_0^\infty(\RR^n)=\linspan\{1\}$ and $\Val_n^\infty(\RR^n)=\linspan\{\vol_n\}$. 

The Alesker product is a bilinear product on $\Val^\infty(\RR^n)$ uniquely determined as follows: For $K_1,K_2\in\calK^\infty_+(\RR^n)$, consider the valuations $\phi_i(K)=\vol_n(K+K_i)$, $i=1,2$, and set
$$
(\phi_1\cdot\phi_2)(K)=\vol_{2n}\big(\Delta_2(K)+K_1\times K_2\big).
$$
Recall that $\Delta_p:\RR^n\to(\RR^n)^p$ is the diagonal embedding given by $\Delta_p(x)=(x,\dots,x)$. Equipped with the Alesker product, $\Val^\infty(\RR^n)$ is an associative, commutative, graded algebra.

\section{Genealogy of the conjectures}
\label{s:genealogy}

In this section we will show that Conjecture \ref{conj2} implies Conjecture \ref{conj1} and that the latter in turn implies Schneider's inequality for the volume of higher-order difference body.

\begin{proposition}
\label{pro:DpK}
For each convex body $K\subset\RR^n$ one has $D_pK=\Delta_pK-K^p$.
\end{proposition}

\begin{proof}
By definition, $(x_1,\dots,x_p)\in D_pK$ if and only if there exist $k_0,\dots,k_p\in K$ such that $k_0=k_1+x_1=\cdots=k_p+x_p$ or, equivalently,
$$
(x_1,\dots,x_p)=(k_0-k_1,\dots,k_0-k_p)=(k_0,\dots,k_0)-(k_1,\dots,k_p) \in \Delta_pK-K^p.
$$
\end{proof}

\begin{proposition}
\label{pro:stronger1}
Conjecture \ref{conj1}, if true, implies Schneider's inequality \eqref{eq:Schneider}.
\end{proposition}

\begin{proof}
Let $n,p\in\NN$. Comparing the coefficients of $x^ny^{pn}$ in the expansion of $(x+y)^{pn+n}$ on the one hand and in the product of expansions $(x+y)^{pn}(x+y)^n$ on the other, one obtains the Vandermonde identity
\begin{align*}
\sum_{k=0}^n{pn\choose k}{n\choose k}= {pn+n\choose n}.
\end{align*}
Consequently, using Proposition \ref{pro:DpK}, if \eqref{eq:conj1} holds then
\begin{align*}
\vol_{pn}(D_pK)&=\vol_{pn}(\Delta_pK-K^p)\\
&=\sum_{k=0}^n{pn\choose k}V(-\Delta_pK[k],K^p[pn-k])\\
&\leq\sum_{k=0}^n{pn\choose k}{n\choose k}\vol_n(K)^p\\
&={pn+n\choose n}\vol_n(K)^p.
\end{align*}
\end{proof}

\begin{proposition}
\label{pro:stronger2}
Conjecture \ref{conj2} implies Conjecture \ref{conj1}.
\end{proposition}

\begin{proof}
Similarly as before, comparing the coefficients of $x^ky^{pn-k}$ in the expansion of $(x+y)^{pn}$ and in the product of expansions $((x+y)^n)^{p}$, one gets the Vandermonde identity
\begin{align*}
\sum{n\choose k_1}\cdots{n\choose k_p}={pn\choose k},
\end{align*}
where the sum runs over all $k_1,\dots,k_p\in\NN_0$ such that $k_1+\cdots+k_p=k$. Consequently, assuming \eqref{eq:conj2} holds, the left-hand side of \eqref{eq:conj1} can be estimated as follows:
\begin{align*}
&V(-\Delta_pK[k],K^p[pn-k])\\
&\quad=V\big(-\Delta_pK[k],(\iota_1K+\cdots+\iota_pK)[pn-k]\big)\\
&\quad=\sum_{j_1,\dots,j_p}{pn-k\choose j_1,\dots,j_p}V(-\Delta_pK[k],\iota_1K[j_1],\dots,\iota_pK[j_p])\\
&\quad=\sum_{k_1,\dots,k_p}\frac{(pn-k)!}{(n-k_1)!\cdots(n-k_p)!}V(-\Delta_pK[k],\iota_1K[n-k_1],\dots,\iota_pK[n-k_p])\\
&\quad\leq{n\choose k}{pn\choose k}^{-1}\vol_n(K)^p\sum_{k_1,\dots,k_p}{n\choose k_1}\cdots{n\choose k_p}\\
&\quad={n\choose k}\vol_n(K)^p.
\end{align*}
\end{proof}

\begin{remark}
\label{rem:symmetry}
 Let $m_0+\cdots+m_p=pn$. Observe that the expression
$$
V(-\Delta_pK[m_0],\iota_1K[m_1],\dots,\iota_pK[m_p])
$$
is symmetric in $m_0,\dots,m_p$. Indeed, the symmetry in $m_0$ and $m_1$ follows by applying the linear map $T:(\RR^n)^p\to (\RR^n)^p$ given by
$$
T(X_1,\dots,X_p)=(-X_1,X_2-X_1,X_3-X_1,\dots,X_p-X_1)
$$
which satisfies $|\det T|=1$, $T\circ-\Delta_p=\iota_1$, $T\circ\iota_1=-\Delta_1$, and $T\circ\iota_i=\iota_i$ for $i>1$. Similarly, the map $S:(\RR^n)^p\to (\RR^n)^p$ given by
$$
S(X_1,\dots,X_p)=(X_1,\dots,X_{i-1},X_j,X_{i+1},\dots,X_{j-1},X_i,X_{j+1},\dots,X_p).
$$
yields the symmetry in $m_i$ and $m_j$ for $0<i<j$.
\end{remark}

\begin{remark}
\label{rem:reduction}
Some cases of Conjecture \ref{conj2} are either trivial or can be reduced to other cases.
\begin{enuma}
\item If $\interior K=\emptyset$, then both sides of \eqref{eq:conj2} are clearly equal to zero.

\item If $k=0$, then necessarily $k_i=0$ for all $i$ and an easy application of Lemma \ref{lem:exact} yields
\begin{align*}
V(\iota_1K[n],\dots,\iota_pK[n])=\frac{(n!)^p}{(pn)!}\vol_n(K)^p.
\end{align*}
Indeed, we consider the exact sequence $0\longrightarrow \RR^n\stackrel{\iota_i}{\longrightarrow} \RR^{pn}\stackrel{\pi_i}{\longrightarrow} (\iota_i\RR^n)^\perp\longrightarrow 0$, where $\pi_i$ is the orthogonal projection, and observe that for $j\neq i$, $\pi_i\circ\iota_j$ is the inclusion into the $j$-th factor of $(\iota_i\RR^n)^\perp=\RR^{(p-1)n}$. The claim then follows by induction on $p$. By Remark \ref{rem:symmetry}, one also has  equality in \eqref{eq:conj2} if $k_i=n$ for some $i$ (in which case $k=n$ and $k_j=0$ for $j\neq i$).
\item If $k_i=0$, then one does not necessarily have equality in \eqref{eq:conj2} but the same argument as in item (b) allows us to inductively reduce these cases of Conjecture \ref{conj2} to those satisfying $k_i>0$ for all $i$. To see this, we observe that one also has $\pi_i\circ\Delta_p=\Delta_{p-1}$. By Remark \ref{rem:symmetry}, the same applies if $k=n$.

\end{enuma}
\end{remark}

\section{Proof for anti-blocking bodies}

In this section we will prove our main result, namely, that Conjecture \ref{conj2} is true for anti-blocking bodies.

\begin{lemma}
\label{lem:union}
Let $K\subset\RR^n$ and $L\subset\RR^{pn}$ be anti-blocking convex bodies. Then
\begin{align}
\label{eq:union}
L-\Delta_pK=\bigcup_{\tau,\sigma}(L_\tau-\Delta_pK_\sigma),
\end{align}
where the union is over all subsets $\tau\subset\{1,\dots,pn\}$ and $\sigma\subset\{1,\dots,n\}$ such that
\begin{align}
\label{eq:p}
|\tau\cap\{i,i+n,\dots,i+(p-1)n\}|+|\sigma\cap\{i\}|\leq p,\quad i=1,\dots, n.
\end{align}
Moreover, the union on the right-hand side of \eqref{eq:union} is disjoint.
\end{lemma}

\begin{proof}
By \eqref{eq:Ksigma}, the equation \eqref{eq:union} holds if the union is taken over all subsets $\tau\subset\{1,\dots,pn\}$ and $\sigma\subset\{1,\dots,n\}$. Clearly, for each $\sigma,\tau$ and each $i$, one has
\begin{align}
\label{eq:p+1}
|\tau\cap\{i,i+n,\dots,i+(p-1)n\}|+|\sigma\cap\{i\}|\leq p+1.
\end{align}
Consider $\tau,\sigma$ such that there is equality in \eqref{eq:p+1} for some $i$. Assume first that $i=n$. Take any $x=(x_1,\dots,x_n)\in K_\sigma$ and $y=(y_1,\dots,y_{pn})\in L_\tau$. One has $x_n,y_n,y_{2n},\dots,y_{pn}>0$. Denote $\zeta=\min\{x_n,y_n,y_{2n},\dots,y_{pn}\}$ and $z=(0,\dots,0,\zeta)\in\RR^n$. Then
\begin{align*}
y-\Delta_p x=(y-\Delta_pz)-\Delta_p(x-z)\in  L_{\tilde\tau}-\Delta_pK_{\tilde \sigma},
\end{align*}
where $|\tilde\tau\cap\{n,2n,\dots,pn\}|+|\tilde\sigma\cap\{n\}|\leq p$. Similarly we proceed with other indices $i$ and so \eqref{eq:union} follows.

To prove that the union is disjoint, assume there are distinct pairs $\sigma,\tau$ and $\sigma',\tau'$ satisfying \eqref{eq:p} and points $x\in K_\sigma$, $y\in L_\tau$, $x'\in K_{\sigma'}$, and $y'\in L_{\tau'}$ such that $y-\Delta_px=y'-\Delta_px'$. Then there is  $i\in\{1,\dots,n\}$ such that the pairs of sets
$$\sigma\cap\{i\},\tau\cap\{i,i+n,\dots,i+(p-1)n\}\quad\text{and}\quad\sigma'\cap\{i\},\tau'\cap\{i,i+n,\dots,i+(p-1)n\}
$$
are distinct.  Without loss of generality, assume again $i=n$. Since $y-y'=\Delta_p(x-x')\in\Delta_p\RR^n$ we in particular have
\begin{align*}
y_n-y_n'=y_{2n}-y_{2n}'=\cdots=y_{pn}-y_{pn}'=x_n-x_n'.
\end{align*}
We will distinguish four cases. First, if $n\in\sigma$ and $n\notin\sigma'$, we have $x_n>0=x_n'$. Consequently, $y_{jn}-y'_{jn}>0$ for $j=1,\dots,p$. But because of \eqref{eq:p}, there is $k$ with $y_{kn}=0$ and we have $-y'_{kn}>0$, a contradiction. The second case $n\notin\sigma$ and $n\in\sigma'$ is analogous. Third, if $n\in\sigma$ and $n\in\sigma'$, because of \eqref{eq:p} and because $\tau\cap\{n,2n,\dots,pn\}\neq\tau'\cap\{n,2n,\dots,pn\}$ there is $k$ with $y_{kn}=0<y_{kn}'$ and hence $y_{jn}-y'_{jn}=-y_{kn}'<0$ for $j=1,\dots,p$. But by \eqref{eq:p} again, there is $l\neq k$ with $y_{ln}'=0$ and thus $y_{ln}<0$, a contradiction. Finally, if $n\notin\sigma$ and $n\notin\sigma'$, we have $x_n=x'_n=0$ and thus $y_{jn}=y'_{jn}$ for $j=1,\dots,p$ which is in contradiction with $\tau\cap\{n,2n,\dots,pn\}\neq\tau'\cap\{n,2n,\dots,pn\}$.

\end{proof}

\begin{proof}[Proof of Theorem \ref{thm:main}]
First, let $K,L^{(1)},\dots,L^{(p)}\subset\RR^n$ be anti-blocking convex bodies and put $L=\sum_{i=1}^p\iota_iL^{(i)}$. By Lemma \ref{lem:union},
\begin{align*}
\vol_{pn}(L-\Delta_pK)=\sum_{\tau,\sigma}\vol_{pn}(L_\tau-\Delta_pK_\sigma),
\end{align*}
where the sum is over all subsets $\tau\subset\{1,\dots,pn\}$ and $\sigma\subset\{1,\dots,n\}$ such that
\begin{align*}
|\tau\cap\{i,i+n,\dots,i+(p-1)n\}|+|\sigma\cap\{i\}|= p,\quad i=1,\dots, n.
\end{align*}
Note that the subsets $\tau,\sigma$ for which the inequality \eqref{eq:p} is strict for some $i$ do not contribute to the sum; indeed, if $|\tau|$+$|\sigma|<pn$ then $L_\tau-\Delta_pK_\sigma$ is of dimension strictly less than $pn$ and hence $\vol_{pn}(L_\tau-\Delta_pK_\sigma)=0$. In other words,
\begin{align*}
\vol_{pn}\big(-\Delta_pK+\iota_1 L^{(1)}+\cdots+\iota_p L^{(p)}\big)=\sum_{\sigma_0,\dots,\sigma_p}\vol_{pn}\big(-\Delta_pK_{\sigma_0}+\iota_1 L^{(1)}_{\sigma_1}+\cdots+\iota_p L^{(p)}_{\sigma_p}\big),
\end{align*}
where the sum runs over all $p$-uniform covers $\sigma_0,\dots,\sigma_p$ of $\{1,\dots,n\}$. Consequently,
\begin{align}
\label{eq:sumsigma}
\begin{split}
&V(-\Delta_pK[k],\iota_1K[n-k_1],\dots,\iota_pK[n-k_p])\\
&\quad=\sum_{\sigma_0,\dots,\sigma_p}V(-\Delta_pK_{\sigma_0}[k],\iota_1K_{\sigma_1}[n-k_1],\dots,\iota_pK_{\sigma_p}[n-k_p]),
\end{split}
\end{align}
where the sum runs over all $p$-uniform covers $\sigma_0,\dots,\sigma_p$ of $\{1,\dots,n\}$ satisfying $|\sigma_0|=k$ and $|\sigma_i|=n-k_i$, $i=1,\dots,p$. Fix such a $p$-uniform cover. By Remark \ref{rem:reduction} we may assume $\interior K\neq\emptyset$, $0<k<n$, and $0<k_i<n$ for $i=1,\dots,p$. We will assume this for the rest of the proof.

For $i=1,\dots,p$, consider the orthogonal decomposition $\RR^n=E_{\sigma_i}\oplus E_{\sigma_i^c}$ and the corresponding projection $\pi_i:\RR^n\to E_{\sigma_i^c}$. Denote $\pi=\pi_1\times\cdots\times\pi_p$. Then we have the exact sequence
\begin{align*}
0\longrightarrow E_{\sigma_1}\times\cdots\times  E_{\sigma_p}\stackrel{\iota_1\times\cdots\times \iota_p}{\longrightarrow} (\RR^n)^p\stackrel{\pi}{\longrightarrow}E_{\sigma_1^c}\times\cdots\times  E_{\sigma_p^c}\longrightarrow 0.
\end{align*}
Notice that both $E_{\sigma_0}$ and $E_{\sigma_1^c}\times\cdots\times  E_{\sigma_p^c}$ can be identified with $\RR^k$ in a natural way. Then, since $\sigma_0,\dots,\sigma_p$ is a $p$-uniform cover, the map $\pi\circ(-\Delta_p)|_{E_{\sigma_0}}$ is just a permutation linear operator on $\RR^k$; in particular, it has determinant $\pm1$. Consequently, Lemma \ref{lem:exact} and the same inductive argument as in Remark \ref{rem:reduction} (b) imply
\begin{align*}
&V\big(-\Delta_pK_{\sigma_0}[k],\iota_1K_{\sigma_1}[n-k_1],\dots,\iota_pK_{\sigma_p}[n-k_p]\big)\\
&\quad={pn\choose k,n-k_1,\dots,n-k_p}^{-1}\vol_k(K_{\sigma_0})\prod_{i=1}^p\vol_{n-k_i}(K_{\sigma_i}).
\end{align*}

The convex body
$$
\hat K=\big\{(\epsilon_1x_1,\dots ,\epsilon_nx_n)\in\RR^n\mid x\in K,\epsilon\in\{-1,1\}^n\big\}
$$
satisfies $0\in\interior \hat K$ and $\vol_{|\sigma|}(\hat K\cap E_\sigma)=2^{|\sigma|}\vol_{|\sigma|}(K_\sigma)$ for any $\sigma\subset\{1,\dots,n\}$. Hence, applying Theorem \ref{thm:dualBT} to $\hat K$ we obtain
\begin{align*}
&V(-\Delta_pK_{\sigma_0}[k],\iota_1K_{\sigma_1}[n-k_1],\dots,\iota_pK_{\sigma_p}[n-k_p])\\
&\quad=\frac1{2^{pn}}{pn\choose k,n-k_1,\dots,n-k_p}^{-1}\vol_k(\hat K\cap E_{\sigma_0})\prod_{i=1}^p\vol_{n-k_i}(\hat K\cap E_{\sigma_i})\\
&\quad\leq\frac1{2^{pn}}{pn\choose k,n-k_1,\dots,n-k_p}^{-1}\frac{(n!)^p}{k!(n-k_1)!\cdots(n-k_p)!}\vol_n(\hat K)^p\\
&\quad =\frac{(n!)^p}{(pn)!}\vol_n(K)^p.
\end{align*}
The right-hand side of this inequality clearly does not depend on the fixed $p$-uniform cover. Since there are ${n\choose n-k,k_1,\dots,k_p}$ such covers, plugging the inequality to \eqref{eq:sumsigma} we finally obtain
\begin{align*}
&V(-\Delta_pK[k],\iota_1K[n-k_1],\dots,\iota_pK[n-k_p])\\
&\quad\leq {n\choose n-k,k_1,\dots,k_p}\frac{(n!)^p}{(pn)!}\vol_n(K)^p\\
&\quad={n\choose k}{k\choose k_1,\dots,k_p}\frac{(n!)^p}{(pn)!}\vol_n(K)^p.
\end{align*}

As for the characterization of the equality cases, observe first that an anti-blocking convex body $K$ is a simplex if and only if $K=\conv\{0,c_1e_1,\dots,c_ne_n\}$ for some $c_1,\dots,c_n>0$. So if this is the case, it clearly holds for any 1-uniform cover $\rho_0,\dots,\rho_p$ of $\{1,\dots,n\}$ that $K=\conv\{K\cap E_{\rho_i}\mid i=0,\dots, p\}$. By Theorem \ref{thm:equality}, one then has equality in Theorem \ref{thm:dualBT} for any $p$-uniform cover $\sigma_0,\dots,\sigma_p$ and consequently equality in \eqref{eq:conj2}.

Assume conversely that equality holds in \eqref{eq:conj2}. Then one has
$$
\vol_k(\hat K\cap E_{\sigma_0})\prod_{i=1}^p\vol_{n-k_i}(\hat K\cap E_{\sigma_i})=\frac{(n!)^p}{k!(n-k_1)!\cdots(n-k_p)!}\vol_n(\hat K)^p
$$
for any $p$-uniform cover $\sigma_0,\dots,\sigma_p$ of $\{1,\dots,n\}$ satisfying $|\sigma_0|=k$ and $|\sigma_i|=n-k_i$, $i=1,\dots,p$. Observe that the $1$-uniform cover induced by such a cover is $\sigma_0^c,\dots,\sigma_p^c$. Therefore, according to Theorem \ref{thm:equality}, for any 1-uniform cover $\rho_0,\dots,\rho_p$ of $\{1,\dots,n\}$ with $|\rho_0|=n-k$ and $|\rho_i|=k_i$, $i=1,\dots,p$, it holds that $\hat K=\conv\{\hat K\cap E_{\rho_j}\mid j=0,\dots,p\}$ and consequently also
\begin{align*}
K=\conv\{ K\cap E_{\rho_j}\mid j=0,\dots,p\}.
\end{align*}
Fix one such 1-uniform cover and assume $K$ is not a simplex. Choose a set $\alpha\in\{\rho_0,\dots,\rho_p\}$ such that $K\cap E_\alpha$ is not a simplex. Observe that this is possible since, by the characterization of anti-blocking simplices mentioned above, $\conv\{K\cap E_\zeta,K\cap E_\eta\}$ is a simplex if $K\cap E_\zeta$ and $K\cap E_\eta$ are simplices. Choose further a smallest subset $\alpha'\subset\alpha$ such that $K\cap E_{\alpha'}$ is not a simplex. Notice that $|\alpha'|\geq 2$. Pick arbitrary $a\in\alpha'$ and $b\in\alpha^c$. Then $K\cap E_{\alpha'\setminus\{a\}}$ is a simplex. Consider another 1-uniform cover $\mu_0,\dots,\mu_p$ with $|\mu_0|=n-k$ and $|\mu_i|=k_i$, $i=1,\dots,p$, that contains $\beta:=(\alpha\setminus\{a\})\cup\{b\}$. Let $\gamma\in\{\mu_0,\dots,\mu_p\}$ be the (unique) set containing $a$. Then
\begin{align*}
K\cap E_{\alpha'}&=\conv\{ K\cap E_{\mu_j}\mid j=0,\dots,p\}\cap E_{\alpha'}\\
&=\conv\{ (K\cap E_{\alpha'})\cap E_{\mu_j}\mid j=0,\dots,p\}\\
&=\conv\{K\cap E_{\alpha'}\cap E_\beta,K\cap E_{\alpha'}\cap E_\gamma\}\\
&=\conv\{K\cap E_{\alpha'\setminus\{a\}},K\cap E_{\{a\}}\}
\end{align*}
is a simplex which is in contradiction with the choice of $\alpha'$ and thus finishes the proof. 
\end{proof}

\section{Relation to the Alesker product}

\label{s:Alesker}

\begin{proposition}
\label{pro:equiv}
Conjecture \ref{conj3} is equivalent to the first part of Conjecture \ref{conj2} (without characterization of the equality cases).
\end{proposition}

\begin{proof}
Let $K\in\calK^\infty_+(\RR^n)$ and $n=k_0+\cdots+k_p$. Recall that that the cokernel of a linear map $f:X\to Y$ is defined as $\coker f=Y/f(X)$. According to \cite[Proposition 5.4]{KW22}, the Alesker product of valuations $\phi_i=(\Cdot[k_i],K[n-k_i])\in\Val^\infty(\RR^n)$, $i=0,\dots,p$, can be expressed as
$$
\phi_0\cdots\phi_p=\frac{k_0!\cdots k_p!(pn)!}{(n!)^{p+1}}V_{\coker\Delta_{p+1}}\big(f_0(K)[n-k_0],\dots,f_p(K)[n-k_p]\big)\vol_n,
$$
where $f_i:\RR^n\to\coker\Delta_{p+1}$ is the composition of the inclusion $\iota_i:\RR^n\to(\RR^n)^{p+1}$ into the $i$-th factor with the canonical projection $\pi:(\RR^n)^{p+1}\to\coker\Delta_{p+1}$ and the mixed volume $V_{\coker\Delta_{p+1}}$ corresponds to the Lebesgue measure $\vol_{\coker\Delta_{p+1}}$ on $\coker\Delta_{p+1}$ induced by the exact sequence
\begin{align}
\label{eq:exact}
0\longrightarrow\RR^n\stackrel{\Delta_{p+1}}{\longrightarrow} (\RR^n)^{p+1}\stackrel{\pi}{\longrightarrow}\coker\Delta_{p+1}\longrightarrow 0.
\end{align}
The linear map $g:(\RR^n)^{p+1}\to(\RR^n)^p$ given by
$$
g(x_0,\dots,x_p)=(x_1-x_0,\dots,x_p-x_0)
$$
clearly descends to an isomorphism $\overline g:\coker\Delta_{p+1}\to(\RR^n)^p$ satisfying $\overline g\circ\pi=g$. We claim that the pushforward of the Lebesgue measure on $\coker\Delta_{p+1}$ under $\overline g$ is precisely the canonical Lebesgue measure on $\RR^{pn}$. Indeed, the linear map $s=(0,\overline g):\coker\Delta_{p+1}\to(\RR^n)^{p+1}$ splits the exact sequence \eqref{eq:exact} so that we have $(\RR^n)^{p+1}=\im\Delta_{p+1}\oplus s(\coker\Delta_{p+1})$ and therefore $(\Delta_{p+1})_*\vol_n\otimes s_*\vol_{\coker\Delta_{p+1}}=\vol_{(p+1)n}$. Then it is easily seen using the Fubini theorem that $s_*\vol_{\coker\Delta_{p+1}}$ is the canonical Lebesgue measure on $\{0\}\times\RR^{pn}$ or, equivalently, $\overline g_*\vol_{\coker\Delta_{p+1}}$ is the canonical Lebesgue measure on $\RR^{pn}$. Consequently,
\begin{align*}
&V_{\coker\Delta_{p+1}}\big(f_0(K)[n-k_0],\dots,f_p(K)[n-k_p]\big)\\
&\quad=V\big(\overline g(f_0(K))[n-k_0],\dots,\overline g(f_p(K))[n-k_p]\big)\\
&\quad=V\big(g(i_0(K))[n-k_0],\dots,g(i_p(K))[n-k_p]\big)\\
&\quad=V(-\Delta_pK[n-k_0],\iota_1K[n-k_1],\dots,i_pK[n-k_p]),
\end{align*}
proving the equivalence for convex bodies from the class  $\calK^\infty_+(\RR^n)$. The general case of Conjecture \ref{conj2} follows by continuity.
\end{proof}

\begin{remark}
\begin{enuma}
\item The quantity $V_{\coker\Delta_{p+1}}\big(f_0(K)[n-k_0],\dots,f_p(K)[n-k_p]\big)$ is a special case of the mixed volume of rank $p$ defined by the author and Wannerer \cite[Definition 1.3]{KW22}. The proof of Proposition \ref{pro:equiv} together with Propositions \ref{pro:stronger1} and \ref{pro:stronger2} thus relates higher-rank mixed volumes to higher-order difference bodies.
\item Observe that the right-hand side of \eqref{eq:conj3} is a special case of the Alesker product on the left-hand side, corresponding to the partition $n=0+\cdots+0+n$. This suggests there might exist even finer inequalities between the products corresponding to two general partitions.
\item As pointed out to us by A.~Bernig, the case $p=1$ of \eqref{eq:conj3}, when written as 
$$
\vol(K)\,\phi_0\cdot\phi_1\leq\phi_0(K)\phi_1(K)\vol_n,
$$
is formally analogous to an inequality proven by Dang and Xiao \cite[Theorem 2.9]{DangXiao} for the (extended) Bernig--Fu convolution of any two valuations $\phi_0,\phi_1$ from a certain infinite-dimensional positive cone. One might thus ask what is the largest positive cone $\calP\subset\Val^\infty(\RR^n)$ such that for any $\phi_i\in \calP\cap\Val^\infty_{k_i}(\RR^n)$, $i=0,\dots,n$, with $k_0+\cdots+k_p=n$ and for  either a fixed or any convex body $K\subset\RR^n$ it holds that 
$$
\vol(K)\,\phi_0\cdots\phi_p\leq\phi_0(K)\cdots\phi_p(K)\vol_n.
$$
\end{enuma}
\end{remark}

\bibliographystyle{abbrv}
\bibliography{ref}

\end{document}